\documentclass[12pt]{amsart}

\usepackage[english]{babel}
\usepackage[T1]{fontenc}
\usepackage[top=2.0cm, bottom=2.0cm, left=2.0cm, right=2.0cm,ignoreall]{geometry}
\usepackage{amsmath, amsfonts}
\usepackage{fancyhdr}

  \usepackage[pdftex]{graphicx,color}
  \usepackage[pdftex]{hyperref}
  \hypersetup{pdftitle={},pdfauthor={D. Vu},colorlinks=true}

\usepackage{fancyhdr}

\newcommand{\A}{\mathcal{H }} 

\theoremstyle{definition}

\newtheorem{thm}{Theorem}
\newtheorem{lem}[thm]{Lemma}

\newtheorem{defn}[thm]{Definition}
\newtheorem{conj}[thm]{Conjecture}

\title{Large Matchings with Few Colors}
\author{Neal Bushaw$^\dag$}
\author{Peter Csorba$^\ddag$}
\author{Lindsay Erickson$^\diamond$}
\author{D\'aniel Gerbner$^\ddag$}
\author{Diana Piguet$^\S$}
\author{Ago-Erik Riet$^\dag$}
\author{Tam\'as Terpai$^\ddag$}
\author{Dominik K. Vu$^\dag$}
\address{$^\dag$University of Memphis}
\address{$^\ddag$R\'enyi Institute}
\address{$^\S$Charles University}
\address{$^\diamond$North Dakota State University} 
\thanks{The problems contained in this paper were presented to the authors by Andr\'as Gy\'arf\'as at the first Eml\'ekt\'abla Workshop put on by the J\'anos Bolyai Mathematical Society and the Alfr\'ed R\'enyi Institute of Mathematics. The work contained in this paper was done over the next several days of the work shop. We are grateful to the organizers of this workshop for their time and effort, and in particular, to Andr\'as Gy\'arf\'as for his insightful comments.}



\begin{document}

\begin{abstract}
Let $K_n^r$ denote the complete $r$-uniform hypergraph on $n$
vertices.  A matching $M$ in a hypergraph is a set of pairwise
vertex disjoint edges, and an $s$-colored matching is a matching using edges from at most $s$ colors.  Recent Ramsey-type results rely on lemmas
about the size of monochromatic matchings in hypergraphs.  A
starting point for this study comes from a well-known result of
Alon, Frankl, and Lov\'asz~\ref{thm:AFL}, which states that any edge-coloring of the complete $r$-uniform hypergraph on $n$ vertices with $t$
colors contains a monochromatic matching of size $k$.  A natural extension of this theorem would
find the smallest $n$ such that every
$t$-coloring of $K_n^r$ contains an $s$-colored matching of size
$k$.  It has been conjectured that in every coloring of the edges of
$K_n^r$ with 3 colors there is a 2-colored matching of size at least
$k$ provided that $n\geq
kr+\left\lfloor{\frac{k-1}{r+1}}\right\rfloor$.  The smallest
non-trivial case is when $r=3$ and $k=4$.  We prove that in every
$3$-coloring of the edges of $K_{12}^3$ there is a $2$-colored
matching of size 4.
\end{abstract}
\maketitle

\section*{Introduction}

Throughout this paper we will rely on the following result of Alon,
Frankl, and Lov\'asz\cite{AFL}, solving a conjecture of Erd\H{o}s
from 1978\cite{ERD}.

\begin{thm}[\cite{AFL}]\label{thm:AFL}
Suppose that $n=(t-1)(k-1)+kr$. Then any edge-coloring with $t$ colors of the
complete $r$-uniform hypergraph on $n$ vertices, $K_n^r$,  induces a monochromatic matching of size at least $k$.
\end{thm}

Extending this idea a bit further, Gy\'arf\'as, S\'ark\"ozy, and
Selkow\cite{GSS} introduced an extra parameter
$s\in\{1,2,\ldots,t\}$.  A matching in a hypergraph which uses edges
from at most $s$ colors is called an \emph{s-colored matching}.  It
is natural to ask what is the smallest $n$ such that every $t$-edge
coloring of $K_n^r$ contains a $s$-colored matching of size $k$?
\cite{GY} states the following conjecture for certain values of $1\le s \le
t$ (in particular for $t=3, \ s=2$).

\begin{conj}[\cite{GY}]\label{conj:GSS}
Every $t$-coloring of the edges of $K_n^r$ (a \emph{t-edge-coloring}) contains an $s$-colored matching of size $k$ provided that $n\geq kr+\left\lfloor{\frac{\left(k-1\right)\left(t-s\right)}{1+r+r^2+\cdots r^{s-1}}}\right\rfloor$.  (We also note here that the case $s=1$ is \ref{thm:AFL}; the case $s=t$ is trivial.)
\end{conj}

The case of $s=t-1, \ r=2$ is solved in \cite{GSS}. The smallest open case of \ref{conj:GSS}  is
$t=3$, $s=2$, $r=3$, $k=4$.  With this in mind, our main result is the following:

\begin{thm}\label{thm:main1}Every 3-edge-coloring of the complete 3-uniform hypergraph on 12 vertices, $K_{12}^{3}$, contains a perfect 2-colored matching.\end{thm}

\section*{Proof of Main Result}

Throughout the following, we assume that  $\mathcal{H}$ is a 3-edge-colored $K_{12}^3$ without a 2-colored matching of size 4.

\begin{defn} A set of 6 vertices from $\mathcal{H}$ is called a \emph{$B$-set} if one color is avoided by all edges induced by these vertices and no disjoint pair of edges is monochromatic.\end{defn}

\begin{defn}
Call a set of 7 vertices a $B^+$ set if every 6-vertex subset chosen out of it is a $B$-set avoiding the same color.
\end{defn}

\begin{lem}\label{lem:NoB+-set}
$\mathcal{H}$ contains no $B^+$ set.
\end{lem}

\begin{proof}[Proof of Lemma~\ref{lem:NoB+-set}]
Suppose that we have a $B^+$-set with vertex-set $\{v_1,v_2,\dots,v_7\}$ with $\binom{B}{[3]}$ colored $c_1$ and $c_2$. Assume without loss of generality that $\{v_1,v_2,v_3\}$ is $c_1$, and so $\{v_4,v_5,v_6\}$ must be $c_2$. Replace one vertex of one of the two matched edges by the left-over vertex.  Then $\{v_5,v_6,v_7\}$ is $c_2$, $\{v_2,v_3,v_4\}$ is $c_1$, $\{v_6,v_7,v_1\}$ is $c_2$,  and so $\{v_3,v_4,v_5\}$ is  also $c_1$, a contradiction.
\end{proof}

\begin{defn}
 Call $A\subseteq V(H)$, $|A|=6$ an \emph{$A$-set}, if $A$ contains a monochromatic perfect matching.
\end{defn}

\begin{lem}\label{lem:A-set}
 The complement of an $A$-set is a $B$-set.
\end{lem}

\begin{proof}[Proof of Lemma~\ref{lem:A-set}]
 Let $A$ be an $A$-set containing a monochromatic perfect matching $M_1$ of color $c_1$. Observe that if the complement $B:=V(\mathcal H)\setminus A$ is an $A$-set, then $\mathcal H$ contains a $2$-colored perfect matching consisting of $M_1$ and the monochromatic perfect matching in $B$. So any perfect matching in $B$ is $2$-colored, and does not contain the color $c_1$, otherwise again it would induce together with $M_1$ a $2$-colored perfect matching in $\mathcal H$. This implies that $B$ is a $B$-set in colors $c_2$ and $c_3$.
\end{proof}


\begin{lem}\label{lem:partition} In any balanced partition of $V(\mathcal H) = Y \dot{\cup} Z$, at least one of the partition classes is a $B$-set.
\end{lem}

\begin{proof}[Proof of Lemma~\ref{lem:partition}]
We may assume that neither $Y$ nor $Z$ is an $A$-set, otherwise by Lemma~\ref{lem:A-set} the complement is a $B$-set.
Suppose that $Y$ is not a $B$-set. Then there is a $2$-colored matching $M_1$ in $Y$, say in $c_1$ and $c_2$, and another $2$-colored matching~$M_2$ in $Y$, say in $c_2$ and $c_3$. If $Z$ has a matching~$M_3$ with the same colors as~$M_1$ or as~$M_2$, the matchings $M=M_3\dot \cup M_1$, or $M=M_3\dot \cup M_2$, respectively, induces a 2-colored perfect matching in $\mathcal H$. This implies that the edges of any matching $M_3$ in $Z$  have two distinct colors, $c_1$ and $c_3$. Thus~$Z$ is a $B$-set.
\end{proof}

\begin{lem}\label{lem:NoRainbow} $\A$ contains no $K_5^3$ using all of $\{c_1, c_2, c_3\}$. \end{lem}

\begin{proof} [Proof of Lemma~\ref{lem:NoRainbow}]
Suppose that there is a $K_5^3$ using colors $c_1, c_2, c_3$ spanning a vertex set $X\subseteq V(\mathcal H)$. Take an arbitrary vertex $v\in V(\mathcal H)\setminus X=:Y$. Then the set $X\cup \{v\}$ is not a $B$-set, implying, by Lemma~\ref{lem:partition}, that $Y\setminus  \{v\}$ is a $B$-set. As the choice of vertex $v$ was arbitrary, the same holds for any vertex $v\in Y$.
We label the vertices of $Y$ by pairs of colors in the following way: a vertex $v \in Y$ is assigned
label $\chi(v)=\{c_1,c_2\}$ if the triples in the $B$-set, $Y\setminus \{v\}$, are colored by~$c_1$ or $c_2$.

First we show that only two of the three possible color pairs may
occur. To see this, suppose we have
vertices $v_1,v_2,v_3 \in Y$ with $\chi(v_1)= \{c_1, c_2\}$, $\chi
(v_2)= \{c_1, c_3\}$, and $\chi(v_3)= \{c_2, c_3\}$. As $|Y\setminus
\{v_1,v_2,v_3\}|=4$ there is an edge in $Y\setminus
\{v_1,v_2,v_3\}$ and its color should be either $c_1$ or $c_2$, and
either $c_1$ or $c_3$, and either $c_2$ or $c_3$, a contradiction.

Now, assume, without loss of generality, that there is a set $U \subseteq Y$ with
vertices labeled $\{c_1, c_2\}$ and a set $W \subseteq Y$ with vertices labeled
$\{c_1$, $c_3\}$, and that $|U|>|W|$. First assume that
$W=\emptyset$. In this case $Y=U$ is a $B^+$-set, a contradiction
by Lemma~\ref{lem:NoB+-set}. Thus $W\neq\emptyset$, so
pick a vertex $w\in W$. Then in any partition of $Y\setminus
\{w\}$ into two disjoint edges, there is one edge, say $e$, that is colored
$c_3$. As $|U|\ge 4$ there is a vertex $u\in U\setminus e$. By definition, any edge in $Y\setminus \{u\}$ is colored either $c_1$ or $c_2$, a contradiction with the fact that $e$ is $c_3$.  Thus there is no $K_5^3$ in $\mathcal H$ using all three colors.

\end{proof}

\begin{lem}\label{lem:edge-complete} For any pair of vertices $u, v \in V(\A)$ there exists a monochromatic $K_6^3 \subset \A$ containing $u$ and $v$. \end{lem}

\begin{proof}[Proof of Lemma~\ref{lem:edge-complete}]
Considering any pair of vertices $x,y \in \mathcal H$, we know from
Lemma~\ref{lem:NoRainbow} that $x$ and $y$ are not contained in any
 $K_5^3$ using all of $\{c_1, c_2, c_3\}$. Thus the edges containing $x$ and $y$ are of at
most two colors. As in the proof of Lemma~\ref{lem:NoRainbow}, label the pair $\{x,y\}$ by the pair of colors it induces. If all edges containing $x$ and $y$ are of the same
color, then choose an arbitrary color to complete the pair.  This defines a coloring of the edges of the complete graph $K_{12}^2$.

Let us choose two arbitrary vertices $u, v \in \mathcal H$. By
Theorem~\ref{thm:AFL} there is a monochromatic matching
$M$ of size $3$ consisting of the 2-edges $\{e_1,e_2,e_3\}$ in the graph induced by the ten
vertices $V(\mathcal H)\setminus \{u,v\}$, $K_{10}^{2}$. Suppose that $M$ is
colored by the color-pair $\{c_1, c_2\}$.  We claim that $V(\mathcal
H)\setminus V(M)$ induces a monochromatic $K_6^3$ colored in $c_3$ in the original hypergraph. If this were not the case, there would exist a $3$-edge $E$ of
$\mathcal H$ with color either $c_1$ or $c_2$. Let $\{ v_1, v_2, v_3
\}$ be the set $V(\mathcal H) \setminus (V(M) \cup E)$. Then the
perfect matching $\{E, e_1 \cup v_1, e_2 \cup v_2, e_3 \cup v_3
\}$ would not contain the color $c_3$, contradicting our
assumption on~$\mathcal H$. Thus, $u$ and $v$ are contained in the monochromatic $K^3_6$ colored in $c_3$ on the vertex set $V(\mathcal H) \setminus V(M)$.

\end{proof}

\begin{lem}\label{lem:same-colored} Any two monochromatic $K_6^3 \subset \A$ intersecting in at most $4$ vertices have different colors.
 \end{lem}


\begin{proof}[Proof of Lemma~\ref{lem:same-colored}]
Consider two distinct monochromatic $K_6^3 \subset \A$, $X_1, X_2$,  of the same color, say $c_1$, with $|V(X_1)\cap V(X_2)|\leq 4$. If $|V(X_1) \cap V(X_2)| \leq 3$, then they span at least $9$ vertices and we would have $3$ disjoint edges of the same color. 
Therefore, we may assume that $|V(X_1) \cap V(X_2)| = 4$. Consider
then the $4$ vertices not contained in either $V(X_1)$ or $V(X_2)$,
$\bar{X} = V(\A) \setminus \big(V(X_1) \cup V(X_2)\big)$. Observe
that no edge $e$ with $|e \cap \bar{X}| = 1$ is $c_1$, otherwise we
would have $9$ vertices forming three disjoint edges colored in
$c_1$. Therefore any such $e$ must be colored either $c_2$ or $c_3$,
and as $|\bar{X}| =4$ they induce a perfect matching colored only in
$c_2$ and $c_3$ which is forbidden in $\mathcal H$. Thus the two
distinct monochromatic $K_6^3$ must have different colors.

\end{proof}

\begin{lem}\label{lem:intersection-complete} Any two distinct monochromatic $K_6^3$ of different colors intersect in exactly 2 vertices.\end{lem}

\begin{proof}[Proof of Lemma~\ref{lem:intersection-complete}]
First, observe that trivially two monochromatic $K_6^3$, of
different colors, say a $c_1$, $X_1$ and a $c_2$, $X_2$ cannot meet
in more than $2$ vertices.  Otherwise any edge contained in their intersection must be of both colors. Also, $X_1$ and $X_2$ must intersect, otherwise $V(X_1)\cup V(X_2)$ induces a $2$-colored perfect matching.  This produces a contradiction with the structure of $\mathcal H$

It remains to show that $X_1$ and $X_2$ cannot intersect in only
one vertex. If this were the case, then by Lemma~\ref{lem:edge-complete} any vertex $v\in V(\mathcal H)\setminus \big(V(X_1)\cup V(X_2)\big)$ is covered by a monochromatic
$K_6^3$, $X$.  As $|V(X)\cap \big(V(X_1)\cup V(X_2)\big)|=5$, the complete hypergraph  $X$
intersects $X_1$ or $X_2$ in at least $3$ vertices, hence its
color is either $c_1$ or $c_2$, say $c_1$. 
 Then $|X \cap X_1|\ge 5$ by Lemma~\ref{lem:same-colored}. We choose two edges colored $c_2$  from $X_2$, and an edge colored $c_1$ containing $v$ from $X$. One can easily see that
the remaining vertices are in $X_1$, hence the edge containing them is $c_1$. This yields a 2-colored matching, which is a contradiction.

\end{proof}

\begin{defn} Call an edge-colored $3$-uniform hypergraph a \emph{disk} if it contains three monochromatic complete subgraphs on six vertices, $X_1$, $X_2$, $X_3$, each colored a different color such that $|V(X_i)\cap V(X_j)|=2$ for $i,j = 1,2,3$ and $V(X_1) \cap V(X_2) \cap V(X_3) = \emptyset$.
\end{defn}

\begin{lem}\label{lem:disk} $\A$ is a disk. \end{lem}

\begin{proof}[Proof of Lemma~\ref{lem:disk}]
Choose any pair of vertices $u_1,v_1\in V(\mathcal H)$.  By Lemma~\ref{lem:edge-complete}, there is a monochromatic complete hypergraph $X_1$ of size six covering $u_1$ and $v_1$. Choose a pair of vertices $u_2,v_2\in V(\mathcal H)\setminus V(X_1)$. By Lemma~\ref{lem:edge-complete}, there is a monochromatic complete hypergraph $X_2$ covering $u_2,v_2$. By Lemma~\ref{lem:same-colored}, the complete hypergraphs $X_1$ and $X_2$ have different colors, and by Lemma~\ref{lem:intersection-complete}, we have $|V(X_1)\cap V(X_2)|=2$. Thus $|V(\mathcal H)\setminus \big(V(X_1)\cup V(X_2)\big)|=2$. Let $u_3,v_3$ be the two vertices from $V(\mathcal H)\setminus \big(V(X_1)\cup V(X_2)\big)$. Again, by Lemma~\ref{lem:edge-complete}, there is a monochromatic complete hypergraph $X_3$ covering $u_3,v_3$. By Lemma~\ref{lem:same-colored}, the color of the edges of $X_3$ is different from the color of $X_1$ and of $X_2$. By Lemma~\ref{lem:intersection-complete}, the hypergraphs $X_3$ and $X_i$, $i=1,2$, intersect in exactly two vertices, hence $\mathcal H$ is a disk.
\end{proof}

\begin{lem}\label{lem:disk-Rainbow} Any disk, $\mathcal D$, contains a $K_5^3$ using all three colors. \end{lem}

\begin{proof}[Proof of Lemma~\ref{lem:disk-Rainbow}]
Since $\mathcal D$ is a disk, it contains $X_1$, $X_2$ and $X_3$ colored $c_1$, $c_2$ and $c_3$, respectively. Let $v_1,v_2\in V(X_1)\cap V(X_2)$, $v_3,v_4\in V(X_2)\cap V(X_3)$ and $v_5\in V(X_1)\cap V(X_3)$. Then the edge $\{v_1,v_2,v_5\}$ is of color $c_1$, the edge $\{v_1,v_2,v_3\}$ is of color $c_2$, and the edge $\{v_3,v_4,v_5\}$ is of color $c_3$.  Thus the complete 3-uniform hypergraph induced on $\{v_1, \ldots, v_5\}$ uses all three colors.  \end{proof}

\begin{proof}[Proof of Theorem~\ref{thm:main1}]

Let $\mathcal H$ be any edge-colored $K_{12}^3$ which does not contain a perfect matching using at most two colors.  By Lemma~\ref{lem:disk}, $\mathcal H$ is a disk.  Thus by Lemma~\ref{lem:disk-Rainbow} $\mathcal H$ contains a $K_{5}^{3}$ using all three colors.  This is a contradiction to Lemmas~\ref{lem:NoRainbow}, thus no such $H$ exists. This proves Theorem~\ref{thm:main1}
\end{proof}

\begin{thm}
\label{thm:main2}
In every edge-coloring of a $3$-uniform hypergraph $K_{16}^3$ on $16$ vertices in three colors, there is a $2$-colored matching of size $5$.
\end{thm}

\begin{proof}[Proof of Theorem~\ref{thm:main2}]
Let $\chi$ be a given edge-coloring of $\mathcal H=K_{16}^3$ by $3$ colors.
 Set $r=3$, $t=3$ and $k=3$. Then $16\ge 13=(t-1)(k-1)+kr$ implying $\mathcal H$ contains  matching $M_1$ of size $3$ that is monochromatic, say in $c_1$.
 If the induced subgraph$\mathcal H'$ on $V(\mathcal H)\setminus V(M_1)$ contains a $c_1$ edge $e$, then $M_1\cup e$ together with any edge from
 the remaining $4$ vertices form a matching of size $5$ colored with two colors, and we are done.
 So we may assume that there is no $c_1$-colored edge in~$\mathcal H'$.

By Lemma~\ref{lem:NoB+-set}, $V(\mathcal H')$ is not a $B^+$-set
and thus contains a monochromatic matching $M_2$. Then $M=M_1\cup
M_2$ leads to the desired 2-colored matching of size $5$.
\end{proof}

\begin{thm}\label{thm:main3}
 In every 3-edge-coloring of a 3-uniform hypergraph $K_{19}^3$ on 19 vertices, there is a 2-colored matching of size 6.
\end{thm}

\begin{proof}[Proof of Theorem~\ref{thm:main3}]
 This proof follows the proof of Theorem~\ref{thm:main2}. Theorem~\ref{thm:AFL} leads to a monochromatic matching $M$ of size $4$ and in the reaming $7$ vertices either we find an edge of the same color as $M$ or a monochromatic pair of edges.
\end{proof}


\thebibliography{llll}

\bibitem{AFL}{N.~Alon, P.~Frankl, L.~Lov\'asz, The chromatic number of Kneser hypergraphs, \emph{Transactions of the American Mathematical Society} 298 (1986), pp. 359-370.}
\bibitem{ERD}{P. Erd\H{o}s, Problems and results in combinatorial
analysis, Colloq. Internat. Theor. Combin. Rome 1973, Acad. Naz.
Lincei, Rome (1976), pp. 3-17.}
\bibitem{GSS}{A. Gy\'arf\'as, G. S\'ark\"ozy, S. Selkow, Coverings by few monochromatic pieces - a transition between two Ramsey problems,
submitted.}
\bibitem{GY}{A. Gy\'arfas, Large matchings with few colors, \emph{Booklet of First Emlektabla
Workshop} 1 (2010), pp. 4-7.}
\end{document}